\documentclass{article}

\usepackage[T2A]{fontenc}
\usepackage{amsmath, amsthm, amsfonts, amssymb, wasysym, enumerate, tabularx, booktabs, xcolor, metalogo, array, enumitem, tikz, comment, setspace, float}
\usepackage[english]{babel}
\usepackage[textwidth=1.5cm,color=mnt!30,linecolor=red,textsize=footnotesize]{todonotes}

\newtheorem{thm}{Theorem}

\newtheorem{cnj}[thm]{Conjecture}
\newtheorem{cor}[thm]{Corollary}
\newtheorem{exa}{Example}

\newtheorem*{fct*}{Fact}

\newtheorem{qst}[thm]{Question}

\newtheorem{res}{Result}

\def\e{{\epsilon}}
\def\g{{\gamma}}

\def\D{{\Delta}}
\def\Th{{\Theta}}

\def\ui{{\"{I}}}

\def\cG{{\cal G}}

\def\sc{{\sf c}}

\def\zZ{{\mathbb Z}} 

\def\capt{{\sf capt}}

\def\deg{{\sf deg}}
\def\diam{{\sf diam}}
\def\dist{{\sf dist}}
\def\gir{{\sf gir}}

\def\rad{{\sf rad}}

\def\hps{{\hat{\pi}^*}}
\def\p{{\pi}}
\def\pc{{\p^\sc}}
\def\ps{{\p^*}}

\definecolor{brwn}{RGB}{140, 70, 20}
\definecolor{gren}{RGB}{  0,140, 10}
\definecolor{mnt}{RGB}{167,230,215}

\newcommand{\nc}[1]{\textcolor{blue}{\sf{#1}}}
\newcommand{\jf}[1]{\textcolor{gren}{\sf{#1}}}
\newcommand{\gh}[1]{\textcolor{brwn}{\sf{#1}}}
\newcommand{\up}[1]{\textcolor{red}{\sf{#1}}}

\begin{document}

\selectlanguage{english}

\title{Cops and Robbers Pebbling in Graphs}

\author{
Nancy E.~Clarke\thanks{
Department of Mathematics and Statistics, 
Acadia University, Wolfville, NS, Canada.
Research support by NSERC grant \#2020-06528.
}
\and
Joshua Forkin\thanks{
Department of Mathematics and Applied Mathematics,
Virginia Commonwealth University, USA
}
\and
Glenn Hurlbert\footnotemark[2]
}

\date{}

\maketitle

\onehalfspacing

\centerline{\it Dedicated to Oleksandr Stanzhytsky, and others like him}
\centerline{\it who are doing less mathematics than usual at this time.}

\begin{abstract}
Here we merge the two fields of Cops and Robbers and Graph Pebbling to introduce the new topic of Cops and Robbers Pebbling.
Both paradigms can be described by moving tokens (the cops) along the edges of a graph to capture a special token (the robber).
In Cops and Robbers, all tokens move freely, whereas, in Graph Pebbling, some of the chasing tokens disappear with movement while the robber is stationary.
In Cops and Robbers Pebbling, some of the chasing tokens (cops) disappear with movement, while the robber moves freely.
We define the cop pebbling number of a graph to be the minimum number of cops necessary to capture the robber in this context, and present upper and lower bounds and exact values, some involving various domination parameters, for an array of graph classes, including paths, cycles, trees, chordal graphs, high girth graphs, and cop-win graphs, as well as graph products.
Furthermore we show that the analogous inequality for Graham's Pebbling Conjecture fails for cop pebbling and posit a conjecture along the lines of Meyniel's Cops and Robbers Conjecture that may hold for cop pebbling.
We also offer several new problems.
\end{abstract}
\begin{quote}
{\bf Keywords:} Cops and Robbers, Graph Pebbling, dominating set
\\
{\bf MSC2020: 05C57, 91A43, 05C69, 90B10}
\end{quote}

\newpage


\section{Introduction}
\label{s:Intro}

There are numerous versions of moving tokens in a graph for various purposes.
Two popular versions are called {\it Cops and Robbers} and {\it Graph Pebbling}.
In both cases we have tokens of one type (C) attempting to capture a token of another type (R), and all token movements occur on the edges of a graph.
In the former instance, all tokens move freely, whereas in the latter instance, type R tokens are stationary and type C movements come at a cost.
In this paper we merge these two subjects to create {\it Cops and Robbers Pebbling}, wherein type R tokens move freely and type C tokens move at a cost.

We define these three paradigms more specifically in Subsection \ref{ss:Defs} below.
The new graph invariant we define to study in this paper is the cop pebbling number of a graph, denoted $\pc(G)$; roughly, this equals the minimum number of pebble-cops necessary to capture the robber in the Cops and Robbers Pebbling paradigm.
In Subsections \ref{ss:Cops}--\ref{ss:OptPebb} we present known results about Cops and Robbers, dominating sets, and optimal pebbling, respectively, that will be used in the sequel.
We record in Subsections \ref{ss:Lower}--\ref{ss:Exact} new theorems on lower bounds, upper bounds, and exact answers for $\pc(G)$, respectively, for a range of graph families, including paths, cycles, trees, chordal graphs, high girth graphs, and cop-win graphs, as well as, in some cases, for all graphs.
The reader already familiar with the prior subjects can skip ahead to Section \ref{s:Main}, where we state and prove our results of the cop pebbling numbers of graphs.
For example, Theorem \ref{t:treebound} proves that the cop pebbling number of $n$-vertex trees is at most $\lceil 2n/3\rceil$, which is tight for paths (and cycles), and Theorem \ref{t:girth} provides an upper bound involving a domination parameter, as a function of girth. 
Section \ref{s:Cart} contains theorems for Cartesian products of graphs.
For example, Theorems \ref{t:CopGrids} and \ref{t:Cubes} give upper and lower bounds on the cop pebbling numbers of grids and cubes, respectively, while Theorem \ref{t:Ktprod} proves that $\pc(G\Box K_t)\le t\pc(G)$ for all $G$.
Furthermore, Theorems \ref{t:GrahamCounter} and \ref{t:WheelProds} show that the analogous inequality from Graham's Pebbling Conjecture fails for cop pebbling.
We finish in Section \ref{s:Open} with some natural questions left open from this work, including a version of Meyniel's Cops and Robbers Conjecture that may hold in the Cops and Robbers Pebbling world, namely that $\pc(G)\le 2n/3+o(n)$ for all $G$.

\subsection{Definitions}
\label{ss:Defs}

In order to describe cop pebbling in Subsection \ref{s:CopPebbling}, we need first to share the definitions and notations from each of the other two areas, as well as what we use from basic graph theory.

We use several standard notations in graph theory, including $V(G)$ for the set of vertices of a graph $G$ (with $n(G)=|V(G)|$), $E(G)$ for its edge set, $\rad(G)$ for its radius, $\diam(G)$ for its diameter, and $\gir(G)$ for its girth, as well as $\deg(v)$ for the degree of a vertex, and $\dist(u,v)$ for the distance between vertices $u$ and $v$.
For a vertex $v$ in a graph $G$, we use the notations $N_d(v)=\{u\mid \dist(u,v)< d\}$ and $N_d[v]=\{u\mid \dist(u,v)\le d\}$.
If $d=1$ we drop the subscript; additionally, we write $N[S]=\cup_{v\in S}N[v]$ for a set of vertices $S$.
We often use $T$ to denote a tree, and set $P_n$, $C_n$, and $K_n$ to be the path, cycle, and complete graph on $n$ vertices, respectively.
(For convenience, we define $C_2=P_2$.)
Thus, the {\it length} (number of edges) of $P_n$ equals $n-1$.

For graphs $G$ and $H$ we define the {\it Cartesian product} $G\Box H$ with vertex set $V(G)\times V(H)$ and edges $(u,v)(w,x)$ if $uw\in E(G)$ and $v=x$ or if $u=w$ and $vx\in E(H)$.
The $d$-{\it dimensional cube} $Q^d$ is defined by $Q^1=P_2$ and $Q^d=Q^{d-1}\Box Q_1$ for $d>1$.

\subsubsection{Graph Pebbling}

A {\it configuration} $C$ of pebbles on a graph $G$ is a function from the vertices of $G$ to the non-negative integers.
Its {\it size} equals $|C|=\sum_{v\in G}C(v)$.
For adjacent vertices $u$ and $v$ with $C(u)\ge 2$, a {\it pebbling step} from $u$ to $v$ removes two pebbles from $u$ and adds one pebble to $v$, while, when $C(u)\ge 1$, a {\it free step} from $u$ to $v$ removes one pebble from $u$ and adds one pebble to $v$.
In the context of moving pebbles, we use the word {\it move} to mean {\it move via pebbling steps}.

The {\it pebbling number} of a graph $G$, denoted $\p(G)$, is the minimum number $m$ so that, from \underline{any} configuration of size $m$, one can move a pebble to any specified {\it target} vertex.
The {\it optimal pebbling number} of a graph $G$, denoted $\ps(G)$, is the minimum number $m$ so that, from \underline{some} configuration of size $m$, one can move a pebble to any specified target vertex. 
We note that in the definitions of the pebbling number and the optimal pebbling number, free moves are not allowed.

\subsubsection{Cops and Robbers}

In Cops and Robbers, the cops are the pebbles, the robber is the target, and the robber is allowed to move.
The Cops and Robbers alternate making moves in {\it turns}.
At each turn, any positive number of cops make one free step, then the robber chooses to make a free step or not.
In Graph Pebbling literature, the activity of moving a pebble to a target is called {\it solving} or {\it reaching} the target; here we use the analogous Cops and Robbers terminology of {\it capturing} the robber.

The {\it cop number} $c(G)$ is defined as the minimum number $m$ so that, from \underline{some} configuration of $m$ cops, it is possible to capture any robber via free steps.
If the cops catch the robber on their $t^{\rm th}$ turn, then we say that the {\it length} of the game is $t$; if the robber wins then the length is infinite.
When the number of cops used is $c(G)$, the {\it capture time} of $G$, denoted $\capt(G)$, is defined to be the length of the game on $G$ when both cops and robbers play optimally.
That is, it equals the minimum (over all cop strategies) of the maximum (over all robber strategies) of the length of the game with $c(G)$ cops on $G$.

\subsubsection{Cop Pebbling}
\label{s:CopPebbling}

The {\it cop pebbling number} $\pc(G)$ is defined as the minimum number $m$ so that, from \underline{some} configuration of $m$ cops, it is possible to capture any robber via pebbling steps.
For example, $c(G)\le n(G)$, since the placement of one cop on each vertex has already caught any robber.
We call an instance of a graph $G$, configuration $C$, and robber vertex $v$ a {\it game}, and say that the cops win the game if they can capture the robber; else the robber wins.
Note that, since we lose a cop with each pebbling step, the cops must catch the robber within at most $|C|-1$ turns to win the game.

We may assume that all graphs are simple.
Because games on $K_1$ are trivial, we will assume that all graph components have at least two vertices.
Additionally, because of the following fact, we will restrict our attention in this paper to connected graphs.

\begin{fct*}
\label{f:sum}
If $G$ has connected components $G_1,\ldots,G_k$ then $\pc(G)=\sum_{i=1}^k\pc(G_i)$.
\end{fct*}

A set $S\subseteq V(G)$ is a {\it distance-$d$ dominating set} if $\cup_{v\in S}N_d[v]=V(G)$.
We denote by $\g_d(G)$ the size of the smallest distance-$d$ dominating set.

\subsection{Cop Results}
\label{ss:Cops}

Here we list the results on Cops and Robbers that will be used to prove our theorems on cop pebbling.
A graph $G$ is {\it cop-win} if $c(G)=1$.
A vertex $u$ in $G$ is called a {\it corner} if there is a vertex $v\not= u$ such that $N[u]\subseteq N[v]$.
We say that $G$ is {\it dismantlable} if either $G$ is a single vertex or there is a corner $u$ such that $G-u$ is dismantlable.
Note that chordal graphs are dismantlable.

\begin{res}
\label{r:copwincharacterization}
\cite{NowaWink} A graph is cop-win if and only if it is dismantlable.
\end{res}

\begin{res}
\label{r:chordalcapture}
\cite{BoGoHaKr}
If $G$ is a chordal graph with radius $r$, then $\capt(G) \leq r$.
\end{res}

This bound is tight. For example, $P_5$ has both radius 2 and $\capt(G) = 2$.

\begin{res}
\label{r:Cayley}
\cite{FranklCayley}
If $G$ is a $d$-regular Cayley graph then $c=c(G)\le \lceil\frac{d+1}{2}\rceil$, and $\capt_c(G) \leq |V(G)|\lceil\frac{d+1}{2}\rceil$.
\end{res}

Result \ref{r:Cayley} yields the following result as a corollary.

\begin{res}
  \label{r:Cube}
\cite{Alspach}
For $d\ge 1$, the $d$-dimensional cube $Q^d$ satisfies $c(Q^d)=\lceil\frac{d+1}{2}\rceil$.
\end{res}

Bonato et al. \cite{BoKiPrGo} found the correct order of magnitude for the capture time of the cube.

\begin{res}
  \label{r:CubeCapt}
\cite{BoKiPrGo}
With $c=c(Q^d)$ we have $\capt_c(Q^d)=\Th(d\lg d)$.
\end{res}

\subsection{Dominating Set Results}
\label{ss:Dom}

In this section we list results on domination that will be used to prove cop pebbling theorems.
The definition of a dominating set immediately yields the following result.

\begin{res}
\label{r:DegDom}
If $G$ is a graph with $n$ vertices and maximum degree $\Delta$ then $\g(G) \ge \frac{n}{\Delta+1}$.
\end{res}

\begin{res}
\label{r:almost}
\cite{BonKemPra}
Almost all cop-win graphs $G$ have $\g(G)=1$.
\end{res}

\begin{res}
\label{r:GridDom}
\cite{GoPiRaTh}
If $G=P_k\Box P_m$ with $16 \leq k \leq m$, then $\g(G) \leq \left\lfloor \frac{(k+2)(m+2)}{5} \right \rfloor - 4$.
\end{res}

Result \ref{r:DegDom} implies that $\g(Q^d)\ge 2^d/(d+1)$.
The following result shows that the actual value is not much greater, asymptotically.

\begin{res}
\label{r:CubeDom}
\cite{KabaPanc}
$\g(Q^d) \sim 2^d/d$.
\end{res}

For integers $k\ge 2$ and $d\ge 2$, define the {\it uniform theta graph} $\Th(k,d)$ as the union of $k$ internally disjoint $xy$-paths, each of length $d$.
For example, $\Th(2,d)\cong C_{2d}$.
We use the following result in Example \ref{e:theta} of Subsection \ref{ss:Upper}, below.

\begin{res}
\label{res:theta}
For all $k\ge 2$ and $d\ge 2$, we have $\g(\Th(k,d)= k\lfloor\frac{d-1}{3}\rfloor + 1 + \lceil((d-1)\bmod 3)/3\rceil$.
\end{res}

\begin{proof}
Let $\Th=\Th(k,d)$, with $x$ and $y$ defined as above.
Let $S$ be a minimum-sized dominating set of $\Th$, $P_1,\ldots,P_k$ be the $k$ $xy$-paths of $\Th$, and define $S_i=S\cap V(P_i)$.

We first claim that, for all $i$, we may assume that if $u,v\in S_i$ and $\dist(u,v)\le 2$, then $v=y$.
This is true because, if it is not so, then we can ``push'' the vertices of $S_i$ toward the direction of $y$ to make it so.

Second, we claim that we may also assume that $x\in S$.
Indeed, suppose that $S$ is a minimum-sized dominating set that does not contain $x$.
If $y\in S$, then use the obvious symmetry that swaps $x$ and $y$ to create a dominating set $S'$ that contains $x$ but not $y$.
Otherwise, suppose that $y\not\in S$.
Now, because $x$ must be dominated, one of its neighbors $x'\in S$; say, $x'\in S_i$.
Choose any $j\not=i$ and let $C=P_i\cup P_j\cong C_{2d}$, with $S_C=S\cap V(C)$.
Note that every vertex of $\Th-C$ is dominated by $S-S_C$.
By ``rotating'' $S_C$ on $C$ once in the appropriate direction, we create a set $S'$ that contains $x$ and dominates $C$, so that $(S-S_C)\cup S'$ is a dominating set of $\Th$ having the same size as $S$.

Finally, using these two claims, we see that $S$ contains $x$, every third vertex of each path, and, in the case that $d\not\equiv 1\pmod 3$, $y$.
Hence, $|S| = 1 + k\lfloor\frac{d-1}{3}\rfloor + \lceil((d-1)\bmod 3)/3\rceil$.
\end{proof}

\subsection{Optimal Pebbling Results}
\label{ss:OptPebb}

Finally, we list the optimal pebbling results we use to establish new cop pebbling theorems.

\begin{res}
\label{r:Optimal}
\cite{BCCMW}
For every graph $G$, $\ps(G)\le \lceil 2n/3\rceil$.
Equality holds when $G$ is a path or cycle.
\end{res}

Fractional pebbling allows for rational values of pebbles.
A {\it fractional pebbling step} from vertex $u$ to one of its neighbors $v$ removes $x$ pebbles from $u$ and adds $x/2$ pebbles to $v$, where $x$ is an rational number such that $0<x\le C(u)$.
The {\it optimal fractional pebbling number} of a graph $G$, denoted $\hps(G)$, is the minimum number $m$ so that, from \underline{some} configuration of size $m$, one can move, via fractional pebbling moves, a sum of one pebble to any specified target vertex.

\begin{res}
\label{r:fracopt}
\cite{HerHesHur,Moews}
For every graph $G$ we have $\ps(G)\ge\lceil\hps(G)\rceil$.
\end{res}

The authors of \cite{HerHesHur} prove that $\hps(G)$ can be calculated by a linear program.
Furthermore, they use this result to show that there is a uniform configuration that witnesses the optimal fractional pebbling number of any vertex-transitive graph; that is, the configuration $C$ defined by $C(v)=\hps(G)/n(G)$ for all $v$ fractionally solves any specified vertex.
From this they prove the following.

\begin{res}
\label{r:trans}
\cite{HerHesHur}
Let $G$ be a vertex-transitive graph and, for any fixed vertex $v$, define $m=\sum_{u\in V(G)}2^{-\dist(u,v)}$.
Then $\hps(G)=n(G)/m$.
\end{res}

For a configuration $C$ on a graph $G$, we say that a vertex $v$ is 2-{\it reachable} if it is possible to move two pebbles to $v$ via pebbling steps.
Then $C$ is 2-{\it solvable} if every vertex of $G$ is 2-reachable.

\begin{res}
\cite{BCCMW}
\label{r:2ReachPath}
If $C$ is a 2-solvable configuration of pebbles on the path $P_n$ then $|C|\ge n+1$.
\end{res}

For a subset $W$ of vertices in a graph $G$ we define the graph $G_W$ to have vertices $V(G_W)=(V(G)-W)\cup\{w\}$ (where $w$ is a new vertex) with edges $xy$ whenever $x,y\in V(G)-W$ and $xy\in E(G)$ and $xw$ whenever $x\in V(G)-W$ and $xz\in E(G)$ for some $z\in W$.
The process of creating $G_W$ from $G$ is called {\it collapsing} $W$.
If $C$ is a configuration on $G$ then we define the configuration $C_W$ on $G_W$ by $C_W(w)=\sum_{z\in W}C(z)$ and $C_W(x)=C(x)$ otherwise.
Note that $|C|=|C_W|$.

\begin{res}
\cite{BCCMW}
\label{r:Collapse}
Let $W$ be a subset of vertices in a graph $G$.
If a configuration $C$ on $G$ can reach the configuration $D$ on $G$ then the configuration $C_W$ on $G_W$ can reach the configuration $D_W$ on $G_W$.
In particular, we have $\ps(G)\ge\ps(G_W)$.
\end{res}

The next two results involve Cartesian products.

\begin{res}
\label{r:OptGridLower}
\cite{PetPorSto} 
For all $k\le m$ we have $\ps(P_k\Box P_m)\ge \frac{5092}{28593}km+O(k+m)\approx 0.1781km$.
\end{res}

This bound is much smaller than the $\frac{2}{7}km+8\approx 0.2857km$ upper bound proved in \cite{GyoKatPap}, which is conjectured to be best possible.

Combining Results \ref{r:fracopt} and \ref{r:trans} yields the lower bound of the following result.
The upper bound places $2^k$ pebbles on each vertex of a distance-$k$ dominating set, with $k$ roughly $d/3$.

\begin{res}
\label{r:OptCube}
\cite{BCCMW,Moews} 
For all $d\ge 1$ we have $(4/3)^d \le \ps(Q^d)\le (4/3)^{d+O(\lg k)}$.
\end{res}


\section{New Theorems on the Cop Pebbling Number}
\label{s:Main}

\subsection{Lower Bounds}
\label{ss:Lower}

The first two theorems give the cop number and optimal pebbling number as lower bounds to the cop pebbling number.

\begin{thm}
\label{t:CopBd}
For any graph $G$, $\pc(G) \ge c(G)$, with equality if and only if $G=K_1$.
\end{thm}

\begin{proof}
Any configuration of cops that can capture the robber via pebbling steps can also capture the robber via free steps.
Thus, $\pc(G) \ge c(G)$.

If $G=K_1$ then $\pc(G)=1=c(G)$, proving one direction.

For the converse, if $\pc(G)=c(G)$ then capturing the robber requires no steps.
That is, if a pebbling capture requires a pebbling step, then the cop lost during the step is irrelevant in the Cops and Robbers capture, which contradicts the equality.
That means that a successful pebbling configuration has no vertex without a pebble; i.e. $\pc(G)=n(G)$.
However, if $n(G)\ge 3$ then, by placing two cops at a vertex $v$ of degree at least two and one cop on each vertex not adjacent to $v$, then we can capture any robber in one step. 
(We especially highlight this statement, recorded as Corollary \ref{c:upper} in Subsection \ref{ss:Upper}, below.) 
Thus $\pc(G)<n$, a contradiction.
If $n(G)=2$ then $G=K_2$ and $\pc(K_2)=2>1=c(K_2)$, a contradiction.
Hence $G=K_1$.
\end{proof}

\begin{thm}
\label{t:OptBd}
For any graph $G$, $\pc(G) \geq \ps(G)$.
\end{thm}

\begin{proof}
Any configuration of $k$ cops, where $k < \ps(G)$, will contain a vertex $v$ which is unreachable. 
The robber can then choose to start and stay on $v$ and thus not be captured. 
\end{proof}

Typically, Theorem \ref{t:OptBd} gives a sharper lower bound on $\pc(G)$ than Theorem \ref{t:CopBd}. 
However, this may not be true for all graphs.

\subsection{Upper Bounds}
\label{ss:Upper}

\begin{thm}
\label{t:Cycle}
The cycle $C_n$ satisfies $\pc(C_n) \le \lceil \frac{2n}{3}\rceil$.
\end{thm}

\begin{proof}
For $C_n$, partition $C_n$ into $\lfloor \frac{n}{3} \rfloor$ copies of $P_3$ and, possibly, an extra $P_1$ or $P_2$. 
Place two cops on the center vertex of each $P_3$, and one cop on each vertex of the remaining one or two vertices.
The robber can only choose to start on one of the copies of $P_3$, where he is next to a pair of cops, and so will be captured on the first move.
Thus $\pc(C_n)\le \lceil \frac{2n}{3} \rceil$.
\end{proof}

\begin{thm}
\label{t:Tree}
If $G$ is a tree then $\pc(G) \le \ps(G)$.
\end{thm}

\begin{proof}
If $G$ is a tree then place $\ps(G)$ cops according to an optimal pebbling configuration $C$.
The robber beginning at some vertex $v$ uniquely defines the subtree $T$ containing $v$ for which every cop in $T$ is on a leaf of $T$, and any leaf of $T$ with no cop is a leaf of $G$.
Thus the robber can never escape $T$.

Now the cops follow the same greedy strategy used by optimal pebbling; i.e. at each state, they make all possible pebbling moves toward $v$.
The new tree $T'$ defined by the robber's new vertex $v'$ and resulting cop pebbles has strictly fewer vertices than $T$.
Hence this process ends after finitely many steps.

Suppose that we allow, for the moment, that pebbles can be split into fractions, so that a fractional pebbling step moves half the amount of pebbles and removes the other half.
Then the process above ends with the capture of the robber at some vertex $v^*$.
Because the set of such pebbling steps is greedy with respect to each temporary robber position, they are greedy with respect to vertex $v^*$.
That is, they are precisely the steps taken to solve $v^*$ from $C$ if the target $v^*$ was known in advance.

Therefore, because $C$ is an optimal pebbling configuration, each of the fractional steps are actually integral; i.e. $C$ captures the robber with pebbling steps.
Hence $\pc(G)\le\ps(G)$.
\end{proof}

We combine Theorem \ref{t:Tree} with Theorem \ref{t:OptBd} above to prove equality in Theorem \ref{t:TreeOpt} of Subsection \ref{ss:Exact} below.

\begin{thm}
\label{t:doubledom}
Let $G$ be a graph, $S$ a subset of its vertices, and define $S'=V-N[S]$.
Then $\pc(G)\le 2|S|+|S'|$.
In particular, $\pc(G)\le 2\g(G)$.
\end{thm}

\begin{proof}
Place two cops on each vertex of $S$ and one cop on each vertex of $S'$.
In order to not be immediately captured, the robber must start in $N[S]-S$, but then is captured in one step by some pair of cops from $S$.
The second statement follows from choosing $S$ to be a minimum dominating set of $G$, since $S'=\emptyset$.
\end{proof}

The authors of \cite{CoDrHeHe} define a {\it Roman dominating set} of $G$ to be a $\{0,1,2\}$-labeling of $V(G)$ so that every vertex labeled $0$ is adjacent to some vertex labeled $2$.
Note that the construction in Theorem \ref{t:doubledom} yields a Roman dominating set by labeling each vertex by its number of cops.
The {\it Roman domination number} $\g_R(G)$ is defined to be the minimum sum of labels of a Roman dominating set.
Hence we obtain the following bound.

\begin{cor}
\label{c:Roman}
    Every graph $G$ satisfies $\pc(G)\le \g_R(G)$.
\end{cor}

To illustrate the improvement of $2|S|+|S'|$ compared to $2\g(G)$, consider the following example.

\begin{exa}
\label{e:doubledom}
For positive integers $m\ge 2k\ge 2$, let $Y=\{y_1,\ldots,y_m\}$ and let $Q=\{Q_1,\ldots,Q_k\}$ be a partition of $Y$ with each part size $|Q_i|\ge 2$.
Define a bipartite graph $G$ with vertices $Y$, $Z=\{z_1,\ldots,z_k\}$, and $x$ as follows.
For each $1\le j\le k$ set $z_j\sim y_i$ if and only if $y_i\in Q_j$.
Also set $x\sim y_i$ for every $1\le i\le m$.
Then $\g(G)=k+1$.
Indeed, since the neighborhoods of each $z_j$ are pairwise disjoint, at least $k$ vertices in $Y\cup Z$ are required to dominate $Z$, one from each $N[z_j]$.
Suppose that $S$ is a dominating set of size $k$.
By the above, $|S\cap N[z_j]|=1$ for all $j$.
But to dominate $x$, some $y_i$ must be in $S$.
Let $y_i\in N(z_j)$; then $y_i$ does not dominate any other $y_{i'}\in N(z_j)$.
Hence $\g(G)\ge k+1$.
It is easy to see that $Z\cup\{x\}$ is a dominating set, so that $\g(G)=k+1$.
By choosing $S=\{x\}$ we have $S'=Z$, so that $\pc(G)\le k+2$, much better than $2\g(G)=2k+2$.
\end{exa}

An obvious corollary of Theorem \ref{t:doubledom} (recorded as Corollary \ref{c:PcDomVertex}, below) is that any graph $G$ with a dominating vertex has $\pc(G)=2$, except $K_1$.
A more interesting corollary is the following.

\begin{cor}
\label{c:upper}
Every graph $G$ satisfies $\pc(G)\le n-\D(G)+1$.
In particular, if $n(G)\le 2$ then $\pc(G)=n$, if $n(G)\ge 3$ then $\pc(G)\le n-1$, and if $n(G)\ge 6$ then $\pc(G)\le n-2$.
\end{cor}

\begin{proof}
Let $v$ be a vertex with $\deg(v)=\D(G)$, set $S=\{v\}$, and apply Theorem \ref{t:doubledom} to obtain the general bound.
Next, it is easy to see that $\pc(K_n)=n$ for $n\le 2$.
Then, a graph with at least three vertices has a vertex of degree at least two, so that $n-\D(G)+1\le n-1$.
Finally, if $\D(G)\ge 3$ then $n-\D(G)+1\le n-2$, while if $\D(G)\le 2$ then $G$ is a path or cycle, for which Theorem \ref{t:pathcycle} yields $\pc(G) = \lceil\frac{2n}{3}\rceil$, which is at most $n-2$ when $n\ge 6$.
\end{proof}

All three conditional bounds in Corollary \ref{c:upper} are tight: for example, $\pc(P_2)=2$, $\pc(P_5) = 4$, and $\pc(P_7) = 5$.
Furthermore, its more general bound of $n-\D(G)+1$ is tight for a graph with a dominating vertex (see Corollary \ref{c:PcDomVertex}).

\begin{thm}
\label{t:induced}
Let $H$ be an induced subgraph of a graph $G$.
Then, for any $s$, if $\pc(H)\le n(H)-s$ then $\pc(G)\le n(G)-s$.
\end{thm}

\begin{proof}
Suppose that $\pc(H)\le n(H)-s$.
Then there is a configuration $C_H$ of $n(H)-s$ cops on $H$ that captures any robber on $H$.
It remains to show that this number of cops can still win when the robber may move off $H$. Define the configuration $C_G$ of $n(G)-s$ cops on $G$ by placing one cop on each vertex of $G-H$ and $C_H(v)$ cops on each vertex $v\in H$.
Because the robber can never move onto a cop, the robber must stay in $H$, thereby being eventually caught by the cops in $H$.
Hence $C_G$ captures any robber on $G$.
\end{proof}

\begin{cor}
\label{c:upper2}
For all $s\ge 2$ there is an $N$ such that every graph $G$ with $n=n(G)\ge N$ has $\pc(G)\le n-s$.
\end{cor}

\begin{proof}
Suppose that $\pc(G)\ge n-s+1$. 
Then Corollary \ref{c:upper} implies that $\D(G)\le s$. 
Consider if $\diam(G)\ge 3s$. 
Then there exists an induced path $P$ of length $3s$ in $G$. 
By Theorem \ref{t:pathcycle} we have $\pc(P_{3s+1})=\lceil 2(3s+1)/3\rceil= 2s+1=n(P)-s$. 
By Theorem \ref{t:induced}, we must have that $\pc(G)\le n-s$, contradicting our assumption that $\pc(G)\ge n-s+1$. 
Thus, we conclude that $\diam(G)<3s$. 
Since there are finitely many (at most $\D(G)^{\diam(G)}$) such graphs, there must be some $N$ such that $\pc(G)\le n-s$ for all $n\ge N$.
\end{proof}

One might be interested in measuring the gap between the size of a graph and its cop pebbling number.
For this we define the {\it cop deficiency} of a graph $G$ to be {\rm\ui}$^\sc(G):=n(G)-\pc(G)$.
Then Theorem \ref{t:induced} and Corollary \ref{c:upper2} can be restated as follows.

\begin{thm}
\label{t:induced2}
Let $H$ be an induced subgraph of a graph $G$.
Then {\rm\ui}$^\sc(G)\ge$ {\rm\ui}$^\sc(H)$.
\end{thm}

\begin{cor}
\label{c:upper3}
For all $s\ge 2$ there is an $N$ such that every graph $G$ with $n=n(G)\ge N$ has {\rm\ui}$^\sc(G)\ge s$.
\end{cor}

\begin{thm}
\label{t:capture}
If $G$ is a cop-win graph with $\capt(G) = t$, then $\pc(G) \leq 2^{t}$.
More generally, if $c(G)=k$ and $\capt_k(G) = t$ then $\pc(G) \leq k2^{t}$.
\end{thm}

\begin{proof}
If $G$ is a cop-win graph with $\capt(G) = t$, then there is some vertex $v$ at which the cop begins and the robber can be caught with free steps in at most $t$ moves. 
If $2^t$ cops are placed on $v$, the cops can use the same capture strategy, and there will be sufficiently many cops for up to $t$ pebbling steps.
Similarly, by placing $2^t$ on each of $c(G)$ cops, there will be sufficiently many cops for up to $t$ rounds of pebbling steps.
\end{proof}

For example, let $T$ be a complete $k$-ary tree of depth $t$.
Then $\capt(T)=t$ by Result \ref{r:chordalcapture}, and so $\pc(T)\le 2^t$.
Theorem \ref{t:capture} is tight for some graphs, as witnessed by any graph $G$ with a dominating vertex (see Corollary \ref{c:PcDomVertex}, below).
It is also tight for any complete $k$-ary tree $T$ of depth two, when $k\ge 3$ (the lower bound follows from Theorem \ref{t:pathcycle} since $T$ contains $P_5$).

\begin{cor}
\label{c:chordalbd}
If $G$ is a chordal graph with radius $r$, then $\pc(G) \leq 2^{r}$.
\end{cor}

\begin{proof}
The result follows from Result \ref{r:chordalcapture} and Theorem \ref{t:capture}.
\end{proof}

\begin{thm}
\label{t:treebound}
If $T$ is an $n$-vertex tree, then $\pc(T) \leq \lceil \frac{2n}{3} \rceil$.
\end{thm}

\begin{proof}
This follows from Theorem \ref{t:Tree} and Result \ref{r:Optimal}.
\end{proof}

We note that Theorems \ref{t:capture} and \ref{t:treebound} can each be stronger than each other, as the following two examples show.
Define the {\it spider} $S(k,d)$ to be the tree having a unique vertex $x$ of degree greater than 2, all $k$ of whose leaves have distance $d$ from $x$.

\begin{exa}
\label{e:treebound1}
For integers $k$ and $d$, the spider $S=S(k,d)$ has $c(S)=1$ and $\capt(S)=d$, with $n=kd+1$.
Thus Theorem \ref{t:capture} yields $\pc(S)\le 2^d$, while Theorem \ref{t:treebound} yields $\pc(S)\le \lceil (2kd+2)/3\rceil$.
Hence one bound is stronger than the other depending on how $k$ compares, roughly, to $3\cdot 2^{d-1}/d$.
\end{exa}

\begin{exa}
\label{e:treebound2}
For integers $k, t\ge 1$, let $T$ be the complete $k$-ary tree of depth $t$.
Then $n(T)=\sum_{i=0}^tk^i=(k^{t+1}-1)/(k-1)$.
Thus Theorem \ref{t:capture} is stronger than Theorem \ref{t:treebound} for $k\ge 3$ and for $k=2$ with $t\ge 2$, while Theorem \ref{t:treebound} is stronger than Theorem \ref{t:capture} when $k=1$ and $t\ge 5$ (because $\capt(P_t)=\lceil t/2\rceil$).
\end{exa}

\begin{thm}
\label{t:girth}
For any positive integer $d$, if $G$ is a graph with $\gir(G) \ge 4d+1$, then $\pc(G) \leq 2^d\g_d(G)$.
\end{thm}

\begin{proof}
Let $S = \{v_1, v_2, \ldots\}$ be a minimum $d$-distance dominating set of $G$, and place $2^d$ cops on each $v_i$.
Suppose the robber starts at vertex $v$.
Since $\gir(G)\ge 4d+1$, we know that $N_d[u]$ is a tree for all $u$.
We write $T_i=N_d[v_i]$ and $T=N_d[v]$ and, for each $u\in T$, denote the unique $vu$-path in $T$ by $P(u)$.

Let $J$ be such that $j\in J$ if and only if $T\cap T_j\not=\emptyset$, and set $Q_j=T\cap T_j$.
Note that $\gir(G)\ge 4d+1$ implies that, for each $j\in J$, the shortest $v_jv$-path $P^*_j$ is unique, which then implies that $Q_j\subseteq P(u)$ for some $u\in T$.
Moreover, by the definition of $S$, we have $\cup_{j\in J}Q_j=T$.

For each $j\in J$, each cop at $v_j$ adopts the strategy to move at each turn toward $v$ along $P^*_j$ until reaching $T$, at which time then moving toward the robber along the unique path in $T$.
This strategy ensures the property that, at any point in the game, if some cop is on vertex $x$ while the robber is on vertex $z$, then the robber can never move to a vertex $y$ for which the unique $yz$-path in $T$ contains $x$ --- which includes $x$ itself.
It also implies that the game will last at most $d$ turns.
Hence, if we suppose that the robber wins the game, then the game lasted exactly $d$ turns and the robber now sits on some vertex $z$.
However, by the definition of $S$, some cop reached $z$ within $d$ turns, which implies by the property just mentioned that the robber cannot move to $z$, a contradiction.
Hence the cops win the game, capturing the robber.
\end{proof}

An obvious corollary (recorded as Corollary \ref{c:PcDomVertex}, below) is that any graph $G$ with a dominating vertex has $\pc(G)=2$.

We remark that Theorem \ref{t:girth} applies to trees.
$P_5$ is an example for which this bound is tight. 
In the case of the spider $S(k,2)$, this bound is significantly better than Corollary \ref{c:upper} when $k$ is large.
The case $d=1$ yields the same upper bound of $2\g(T)$ from Theorem \ref{t:doubledom}, which is better than the bound of Theorem \ref{t:treebound} if and only if $\g(T)<\lceil (n-1)/3\rceil$.
Since $\g(T)$ can be as high as $n/2$, both theorems are relevant.
The following example shows that Theorem \ref{t:girth} can be stronger than Theorem \ref{t:treebound} for any $d$.

\begin{exa}
\label{e:ternary}
For $1\le i\le 3$, define the tree $T_i$ to be the complete binary tree of depth $d$, rooted at vertex $v_i$, and define the tree $T$ to be the union of the three $T_i$s with an additional root vertex adjacent to each $v_i$.
Then $\g_d(T)=d$, and $n=3(2^{d+1}-1)+1$, so that the bound from Theorem \ref{t:girth} is stronger than the bound from Theorem \ref{t:treebound}.
\end{exa}

Theorem \ref{t:girth} can be stronger than other prior bounds as well, as shown by the following example.

\begin{exa}
\label{e:theta}
The uniform theta graph $\Th=\Th(k,2d+1)$ has $n=2kd+2$, $c(\Th)=2$, $\capt_2(\Th)=d$, $\gir(\Th)=4d+2$, $\g(\Th)=k\lfloor 2d/3\rfloor + 1 + \lceil(2d\bmod 3)/3\rceil$ (by Result \ref{res:theta}, and $\g_d=2$.
Thus Theorem \ref{t:doubledom} yields an upper bound of roughly $4kd/3$, while Theorems \ref{t:capture} and \ref{t:girth} both yield the upper bound of $2^{d+1}$, which is better or worse than Theorem \ref{t:doubledom} when $k$ is bigger or less than, roughly, $3\cdot 2^{d-1}/d$.
\end{exa}

The following example illustrates the need for stronger bounds than given by Theorem \ref{t:girth}.

\begin{exa}
\label{e:McGee}
Consider the $(3,7)$-cage {\it McGee graph} $M$, defined by $V=\{v_i\mid i\in\zZ_{24}\}$, with $v_i\sim v_{i+1}$ for all $i$, $v_i\sim v_{i+12}$ for all $i\equiv 0\pmod{3}$, and $v_i\sim v_{i+7}$ for all $i\equiv 1\pmod{3}$.
We have $\g_1(M)\le 7$\footnote{In fact, $\g_1(M)=7$ --- see House of Graphs: \texttt{https://houseofgraphs.org/graphs/1223}.} (e.g. $\{v_0,v_5,v_9,v_{10},v_{14},v_{15},v_{19}\}$), and so $\ps(M)\le\pc(M)\le 14$ by Theorem \ref{t:girth}.
However, this bound is not tight, as $\pc(M)\le 12$: the vertex set $\{v_i\mid i\equiv 0\pmod{3}\}$ induces a matching of size 4 --- for each edge, place 2 cops on one of its vertices and 1 cop on the other.
Incidentally, this yields $\ps(M)\le 12$; the best known lower bound on $\ps$ comes from Result \ref{r:trans}: $\ps(M) \ge \lceil \hps(M) \rceil = \lceil 64/7 \rceil = 10$.
Hence we are left with a gap in the bounds for $M$: $10\le \ps(M)\le \pc(M)\le 12$.
\end{exa}

\subsection{Exact Results}
\label{ss:Exact}

The following is a corollary of Theorem \ref{t:doubledom}, as well as of Theorem \ref{t:girth}.

\begin{cor}
\label{c:PcDomVertex}
If $G$ is a graph with at least two vertices and a dominating vertex then $\pc(G) = 2$.
\end{cor}

The following is an immediate consequence of Result \ref{r:almost} and Corollary \ref{c:PcDomVertex}.

\begin{cor}
Almost all cop-win graphs $G$ have $\pc(G)=2$.
\end{cor}

\begin{thm}
\label{t:pathcycle}
For all $n\ge 1$ we have $\pc(P_n) = \pc(C_n) = \lceil \frac{2n}{3} \rceil$. 
\end{thm}

\begin{proof}
Result \ref{r:Optimal} and Theorem \ref{t:OptBd} provide the lower bound for both graphs.
The upper bounds for paths and cycles follow from Theorems \ref{t:treebound} and \ref{t:Cycle}, respectively.
\end{proof}

\begin{thm}
\label{t:TreeOpt}
For any tree $T$ we have $\pc(T) = \ps(T)$.
\end{thm}

\begin{proof}
This follows from Theorems \ref{t:OptBd} and \ref{t:Tree}.
\end{proof}

A simple corollary to Theorem \ref{t:TreeOpt} is that any tree with radius 2 and diameter 4 has cop pebbling number 4.
A complete $k$-ary tree of depth 2 with $k\ge 3$ is such an example.


\section{Cartesian Products}
\label{s:Cart}

For ladders ($P_2\Box P_m$) we have the following theorem.

\begin{thm}
\label{t:ladders}
For all $m\ge 1$ we have $\pc(P_2\Box P_m) = m+1$.
\end{thm}

\begin{proof}
Let $G=P_2\Box P_m$ and for each $i\in [m]$ define the vertex subset $E_i=[2]\times\{i\}$.

For the lower bound, we suppose that $C$ can catch any robber on $G$.
In particular, $C$ can catch any robber that uses any particular strategy.
Then we claim that $C$ must be able to move two pebbles to any $E_i$.
Indeed, assume that the robber decides to start at a vertex of $E_i$ and uses the strategy of moving only when forced to (i.e., threatened to be captured on the next round otherwise) and always staying in $E_i$.
Then a pebble catches the robber on $E_i$; without loss of generality, on vertex $(0,i)$.
If the robber didn't move on their turn just prior to being caught, it is because it couldn't: there was a cop on its neighbor vertex $(1,i)$ in $E_i$.
If the capturing move was from $(1,i)$, then those were the two pebbles in $E_i$ already.
Otherwise a second pebble is placed in $E_i$ from the capturing move from a neighboring $E_{i'}$.
Now, by collapsing every $E_i$ we obtain the graph $G'=P_m$, with corresponding collapsed configuration $C'$.
The above argument then shows by Result \ref{r:Collapse} that $C'$ is 2-solvable.
Consequently, Result \ref{r:2ReachPath} proves that $|C|=|C'|\ge m+1$.

\begin{figure}
    \centering
\begin{tikzpicture}
\tikzstyle{every node}=[draw,fill=black,shape=circle,inner sep=2pt];
\def \hit {1}
\def \wid {\hit}
\node (v00) at ({\wid*0},{\hit*0}) {};
\node[label={above:1}] (v01) at ({\wid*0},{\hit*1}) {};
\node[label={below:2}] (v10) at ({\wid*1},{\hit*0}) {};
\node (v11) at ({\wid*1},{\hit*1}) {};
\node (v20) at ({\wid*2},{\hit*0}) {};
\node (v21) at ({\wid*2},{\hit*1}) {};
\node (v30) at ({\wid*3},{\hit*0}) {};
\node[label={above:2}] (v31) at ({\wid*3},{\hit*1}) {};
\node (v40) at ({\wid*4},{\hit*0}) {};
\node (v41) at ({\wid*4},{\hit*1}) {};
\node[label={below:2}] (v50) at ({\wid*5},{\hit*0}) {};
\node (v51) at ({\wid*5},{\hit*1}) {};
\draw[thick] (v00)--(v01);
\draw[thick] (v10)--(v11);
\draw[thick] (v20)--(v21);
\draw[thick] (v30)--(v31);
\draw[thick] (v40)--(v41);
\draw[thick] (v50)--(v51);
\draw[thick] (v00)--(v10)--(v20)--(v30)--(v40)--(v50);
\draw[thick] (v01)--(v11)--(v21)--(v31)--(v41)--(v51);
\end{tikzpicture}
$\qquad\qquad$
\begin{tikzpicture}
\tikzstyle{every node}=[draw,fill=black,shape=circle,inner sep=2pt];
\def \hit {1}
\def \wid {\hit}
\node (v00) at ({\wid*0},{\hit*0}) {};
\node[label={above:1}] (v01) at ({\wid*0},{\hit*1}) {};
\node[label={below:2}] (v10) at ({\wid*1},{\hit*0}) {};
\node (v11) at ({\wid*1},{\hit*1}) {};
\node (v20) at ({\wid*2},{\hit*0}) {};
\node (v21) at ({\wid*2},{\hit*1}) {};
\node (v30) at ({\wid*3},{\hit*0}) {};
\node[label={above:2}] (v31) at ({\wid*3},{\hit*1}) {};
\node (v40) at ({\wid*4},{\hit*0}) {};
\node (v41) at ({\wid*4},{\hit*1}) {};
\node[label={below:2}] (v50) at ({\wid*5},{\hit*0}) {};
\node (v51) at ({\wid*5},{\hit*1}) {};
\node (v60) at ({\wid*6},{\hit*0}) {};
\node[label={above:1}] (v61) at ({\wid*6},{\hit*1}) {};
\draw[thick] (v00)--(v01);
\draw[thick] (v10)--(v11);
\draw[thick] (v20)--(v21);
\draw[thick] (v30)--(v31);
\draw[thick] (v40)--(v41);
\draw[thick] (v50)--(v51);
\draw[thick] (v60)--(v61);
\draw[thick] (v00)--(v10)--(v20)--(v30)--(v40)--(v50)--(v60);
\draw[thick] (v01)--(v11)--(v21)--(v31)--(v41)--(v51)--(v61);
\end{tikzpicture}
    \caption{Examples of the Roman dominating sets used for the ladders in the upper bound of Theorem \ref{t:ladders}, with even $m=6$ ($(r,s,t)= (1,1,0)$) on the left and odd $m=7$ ($(r,s,t)= (1,1,1)$) on the right.}
    \label{f:RomDomLadders}
\end{figure}

For the upper bound, we define $[k]=\{0,1,\ldots,k-1\}$, write $m$ uniquely as $m=4r+2s+t$, with $s,t\in [2]$, and let $V(G)=[2]\times [m]$.
Next define the sets $S=\{(1,0)\}$ when $m$ is even and $S=\{(1,0),(s,m-1)\}$ when $m$ is odd, and $T=\{(j,4i+2j+1)\mid 0\le i\le r, j\in [2], 4i+2j+1\in [m]\}$ --- alternately, $T=\{(0,x)\mid x\equiv 1\pmod{4}\}\cup \{(1,x)\mid x\equiv 3\pmod{4}\}$.
Then $S\cup T$ is a dominating set (see Figure \ref{f:RomDomLadders}).
Now place two cops on each vertex of $T$ and one cop on each vertex of $S$.
It is simple to check that any robber can be captured in one step and that the number of cops in each case equals $m+1$.
(In fact, this is a Roman dominating set.)
\end{proof}

More general grids have cop pebbling numbers linear in their number of vertices, but there is a gap in the bounds for its coefficient.
We use the following result on the Roman domination number of grids.

\begin{res}
\label{r:RomGrid}
\cite{Curro}
For all $k, m\ge 5$ we have $\g_R(P_k\Box P_m)\le \lfloor\frac{2(km+k+m)}{5}\rfloor$.
\end{res}

\begin{thm}
\label{t:CopGrids}
For all $5\le k\le m$ we have $\frac{5092}{28593}km+O(k+m) \le \pc(P_k\Box P_m)\le \lfloor\frac{2(km+k+m)}{5}\rfloor$.
The lower bound also holds for all $1\le k\le m$.
\end{thm}

\begin{proof}
Result \ref{r:OptGridLower} and Theorem \ref{t:OptBd} produce the lower bound, while Result \ref{r:RomGrid} and Corollary \ref{c:Roman} produce the upper bound.
\end{proof}

Note that this is another example of Roman domination giving a better upper bound than just domination (see Result \ref{r:GridDom}).

For cubes, we can use Results \ref{r:Cube} and \ref{r:CubeCapt} with Theorem \ref{t:capture} to obtain the upper bound $\pc(Q^d)\le \lceil\frac{d+1}{2}\rceil d^{\Th(d)}$.
However, by adding extra cops in the Cops and Robbers game, we can reduce the capture time and therefore also reduce the upper bound for cop pebbling.
Still, an exponential gap remains.

\begin{thm}
\label{t:Cubes}
$\left(\frac{4}{3}\right)^d \le \pc(Q^d)\le \frac{2^{d+1}}{d}+o(2^d/d)$.
\end{thm}

\begin{proof}
The lower bound follows from Result \ref{r:OptCube} and Theorem \ref{t:OptBd}.
The upper bound follows from Result \ref{r:CubeDom} and Theorem \ref{t:doubledom}.
\end{proof}

\begin{thm}
\label{t:Ktprod}
For every graph $G$ we have $\pc(G\Box K_t) \leq t\pc(G)$. 
\end{thm}

\begin{proof}
Let $C$ be a configuration of $\pc(G)$ cops on $G$ that can capture any robber.
Define the configuration $C'$ on $G\Box K_t$ by $C'(u,v)=C(u)$ for all $u\in V(G)$ and $v\in V(K_t)$; then $|C'|=t|C|$.
Let $C'_v$ be the restriction of $C'$ to the vertices $V_v=\{(u,v)\mid u\in V(G)\}$.
Then each $C'_v$ is a copy of $C$ on $V_v$.
Now imagine, for any robber on some vertex $(u,v)$, placing a copy of the robber on each vertex $(u',v)$ and maintaining that property with every robber movement.
Then the cops on each $V_v$ will move in unison to catch their copy of the robber in $V_v$, one of which is the real robber.
\end{proof}

A famous conjecture of Graham \cite{Chung} postulates that every pair of graphs $G$ and $H$ satisfy $\p(G\Box H)\le \p(G)\p(H)$.
This relationship was shown by Shiue to hold for optimal pebbling.

\begin{res}
\label{r:OptProd}
\cite{Shiue}
Every pair of graphs $G$ and $H$ satisfy $\ps(G\Box H)\le \ps(G)\ps(H)$.
\end{res}

One might ask whether or not the analogous relationship holds between $\pc(G\Box H)$ and $\pc(G)\pc(H)$.
Theorem \ref{t:Ktprod} shows that this is true for $H=K_2$.
However, the inequality is false in general, as the following theorem shows.
For any graph $G$ define $G^1=G$ and $G^d=G\Box G^{d-1}$ for $d>1$.

\begin{thm}
\label{t:GrahamCounter}
There exist graphs $G$ and $H$ such that $\pc(G\Box H) > \pc(G)\pc(H)$.
\end{thm}

\begin{proof}
Suppose that $\pc(G\Box H)\le \pc(G)\pc(H)$ for all $G$ and $H$.
For fixed $k\ge 2$, let $d\ge 25k^2$, $v\in V(C_k^d)$, and $m=\sum_{u\in V(C_k^d)}2^{-\dist(u,v)}$.
Then $\sqrt{d}>\ln d$, so that $d/\ln d>\sqrt{d}\ge 5k>\frac{2}{\ln (3/2)}k$, which implies that $d^{2k}<(3/2)^d$.
Also $d\ge \sqrt{k/8}+2$, so that $\sqrt{k/8}\ge d-2$.
Thus
\begin{align*}
    \left(\frac{2}{3}\right)^d\sum_{u\in V(C_k^d)}2^{-\dist(u,v)}\
    &\le\ \left(\frac{2}{3}\right)^d\sum_{i=0}^{kd/2}\binom{i+k-1}{k-1}2^{-i}\ 
    \le\ \left(\frac{2}{3}\right)^d\sum_{i=0}^{kd/2}\binom{i+k-1}{k-1}\
    \le\ \left(\frac{2}{3}\right)^d\binom{kd/2+k}{k}\\
    &\le\ \left(\frac{2}{3}\right)^d(kd/2+k)^k/k!\
    \le\ \left(\frac{2}{3}\right)^d(d+2)^k\sqrt{k/2}^k2^{-k}\
    \le\ \left(\frac{2}{3}\right)^d(d+2)^k(d-2)^k\\
    &\le\ \left(\frac{2}{3}\right)^dd^{2k}\
    <\ 1.
\end{align*}
Therefore we would have
\[\pc(P_k^d)\
\le\ \pc(P_k)^d\
\le\ \left(\frac{2}{3}k\right)^d\
=\ \left(\frac{2}{3}\right)^d n(P_k^d)\ 
<\ n(C_k^d)/m\ 
=\ \hps(C_k^d)\
\le\ \hps(P_k^d)\
\le\ \ps(P_k^d),\]
by Fact \ref{r:fracopt} and Theorem \ref{r:trans}.
This, however, contradicts Theorem \ref{t:OptBd}.
\end{proof}

\begin{figure}
    \centering
\begin{tikzpicture}[scale=0.7]
\tikzstyle{every node}=[draw,fill=black,shape=circle,inner sep=2pt];
\def \n {7}
\def \nn {\numexpr \n-1 \relax}
\def \Rad {6}
\def \rad {1}
\def \ang {360/\n}
  \foreach \i in {0,...,\nn} {
    \foreach \j in {0,...,\nn} {
      \node (v\i\j) at ({\Rad*cos(90-\i*\ang)+\rad*cos(90-\j*\ang)}, {\Rad*sin(90-\i*\ang)+\rad*sin(90-\j*\ang)}) {};
    }
    \node (v\i\n) at ({\Rad*cos(90-\i*\ang)}, {\Rad*sin(90-\i*\ang)}) {};
  }
    \foreach \j in {0,...,\nn} {
      \node (v\n\j) at ({\rad*cos(90-\j*\ang)}, {\rad*sin(90-\j*\ang)}) {};
    }
    \node (v\n\n) at (0,0) {};
  \foreach \i in {0,...,\nn} {
    \foreach \j in {0,...,\nn} {
      \draw[thick] (v\i\j)--(v\i\the\numexpr\j+1\relax);
      \draw[thick] (v\i\j)--(v\i\n);
      \draw[thick] (v\i\j)--(v\the\numexpr\i+1\relax\j);
      \draw[thick] (v\i\j)--(v\n\j);
    }
    \draw[thick] (v\i\the\numexpr\n-1\relax)--(v\i0);
  }
    \foreach \j in {0,...,\nn} {
      \draw[thick] (v\n\j)--(v\n\the\numexpr\j+1\relax);
      \draw[thick] (v\n\j)--(v\n\n);
      \draw[thick] (v\the\numexpr\n-1\relax\j)--(v0\j);
    }
    \draw[thick] (v\n\the\numexpr\n-1\relax)--(v\n0);
\end{tikzpicture}
    \caption{The Cartesian product of two wheels $W_8\Box W_8$.}
    \label{f:WheelSquared}
\end{figure}

A more direct example is given by the Cartesian product of wheels.
For $n\ge 4$, define the {\it wheel} $W_n$ by the addition of a dominating vertex $x$ to the cycle $C_{n-1}$, having vertices $v_0, \ldots, v_{n-2}$.
Figure \ref{f:WheelSquared} shows $W_8\Box W_8$.

\begin{thm}
\label{t:WheelProds}
Let $n\ge 4$ and let $G=W_n\Box W_n$.
Then $\pc(G)\le 8$ and if $n\ge 36$ then $\pc(G)=8$.
\end{thm}

\begin{figure}
    \centering
\begin{tikzpicture}[scale=0.9]
\tikzstyle{every node}=[draw,fill=black,shape=circle,inner sep=2pt];
\def \n {13}
\def \nn {\numexpr \n-1 \relax}
\def \k {5}
\def \kk {\numexpr \k-1 \relax}
\def \kkk {\numexpr \kk-1 \relax}
\def \Rad {6}
\def \rad {1}
\def \ang {360/\n}
  \foreach \i in {0,...,\kk} {
    \foreach \j in {0,...,\nn} {
      \node (v\i\j) at ({\Rad*cos(90-\i*\ang)+\rad*cos(90-\j*\ang)}, {\Rad*sin(90-\i*\ang)+\rad*sin(90-\j*\ang)}) {};
    }
    \node (v\i\n) at ({\Rad*cos(90-\i*\ang)}, {\Rad*sin(90-\i*\ang)}) {};
  }
    \foreach \j in {0,...,\kk} {
      \node (v\n\j) at ({\rad*cos(90-\j*\ang)}, {\rad*sin(90-\j*\ang)}) {};
    }
  \foreach \i in {0,...,\kkk} {
    \foreach \j in {0,...,\nn} {
      \draw[thick] (v\i\j)--(v\i\the\numexpr\j+1\relax);
      \draw[thick] (v\i\j)--(v\i\n);
      \draw[thick] (v\i\j)--(v\the\numexpr\i+1\relax\j);
    }
    \draw[thick] (v\i\the\numexpr\n-1\relax)--(v\i0);
  }
    \foreach \j in {0,...,\nn} {
      \draw[thick] (v\the\numexpr\k-1\relax\j)--(v\the\numexpr\k-1\relax\the\numexpr\j+1\relax);
      \draw[thick] (v\the\numexpr\k-1\relax\j)--(v\the\numexpr\k-1\relax\n);
    }
    \draw[thick] (v\the\numexpr\k-1\relax\the\numexpr\n-1\relax)--(v\the\numexpr\k-1\relax0);
    \foreach \j in {0,...,\kkk} {
      \draw[thick] (v\n\j)--(v\n\the\numexpr\j+1\relax);
      \foreach \i in {0,...,\kk} {
        \draw[thick] (v\i\j)--(v\n\j);
      }
    }
    \foreach \i in {0,...,\kk} {
    \draw[thick] (v\i\the\numexpr\k-1\relax)--(v\n\the\numexpr\k-1\relax);
    }
\end{tikzpicture}
    \caption{The subgraph $H_5\subset W_{14}\Box W_{14}$.}
    \label{f:H5}
\end{figure}

\begin{proof}
First we show the upper bound.
Note that because $x$ dominates $W_n$, the vertex $(x,x)\in G$ has eccentricity 2.
By placing $8$ cops on $(x,x)$, the robber must choose a vertex outside of $N[(x,x)]$ to start on; by symmetry let it be $(v_1,v_1)$.
Then we move 2 cops to each of $(x,v_1)$ and $(v_1,x)$.
Because $\{(x,v_1),(v_1,x)\}$ dominates $N[(v_1,v_1)]$, the robber is caught in the next round, regardless of where it moves (or doesn't move) to.
Hence $\pc(G)\le 8$.

Now we prove the lower bound.
Define the subgraph $H_k$ of $G$ to be induced by the vertices $\{(v_i,w)\mid i\in [k], w\in V(W_n)\} \cup \{(x,v_j)\mid j\in [k]\}$.
Suppose that $n\ge 36$ and $7$ pebble-cops are placed on $G$.
Then there is some subgraph $H$ isomorphic to $H_5$ with no pebble-cops in it.
(See Figure \ref{f:H5}.)
Without loss of generality, $H = H_5$.
We say that a vertex is {\it safe} when it is not occupied by a cop and is not a neighbor of a vertex occupied by two cops; it is {\it unsafe} otherwise.
Now assume that a configuration $C$ catches any robber via a sequence of pebbling moves.
Then $C$ catches a robber who begins at vertex $a=(v_2,v_2)$, who only moves when its current vertex becomes unsafe and there is a safe vertex to move to, and who prioritizes not moving to vertices with an $x$ coordinate.

Observe that the distance from $a$ to a vertex outside of $H$ is at least 2, with equality only for the vertex $(x,x)$.
Also note that, when the robber is about to be caught at vertex $z$ on the next cop turn, no vertex of $N[z]$ is safe; that is, every vertex of $N[z]$ not occupied by a cop is next to some vertex with two cops --- this is what we call a {\it local} Roman domination of $N[z]$.

Suppose that the robber is caught at $z=(v_i,v_j)$, for some appropriate $i$ and $j$.
Then $N[z]$ induces the following subgraph of $G$.
\begin{figure}[h]
    \centering
\begin{tikzpicture}
\tikzstyle{every node}=[draw,fill=black,shape=circle,inner sep=2pt];
\def \hit {1}
\def \wid {\hit}
\node[label={left:$(x,v_j)$}] (x1) at ({\wid*0},{\hit*0}) {};
\node[label={[label distance=-0.4cm]above:$(v_{i-1},v_j)$}] (01) at ({\wid*1},{\hit*1}) {};
\node[label={[label distance=-0.4cm]below:$(v_{i+1},v_j)$}] (21) at ({\wid*1},{-1*\hit}) {};
\node[label={below:$z$}] (11) at ({\wid*2},{\hit*0}) {};
\node[label={[label distance=-0.4cm]above:$(v_i,v_{j-1})$}] (10) at ({\wid*3},{\hit*1}) {};
\node[label={[label distance=-0.4cm]below:$(v_i,v_{j+1})$}] (12) at ({\wid*3},{-1*\hit}) {};
\node[label={right:$(v_i,x)$}] (1x) at ({\wid*4},{\hit*0}) {};
\draw[thick] (x1)--(01)--(11)--(10)--(1x)--(12)--(11)--(21)--(x1)--(11)--(1x);
\end{tikzpicture}
\end{figure}

Suppose that $z\not=a$.
Then the robber moved from $a$ to $a'$, which required two cops to move to one of its neighbors, which required at least 4 cops to move from $C$; in fact, exactly 4 cops were used, who all of whom originated from $(x,x)$.
Now there are at most 3 cops who have yet to move.
Because 3 cops are insufficient by themselves to move 2 to a neighbor of $a'$, they must involve the first two cops, moving one from each location (assuming that two of the three as-yet-unmoved cops shared the same vertex) to a neighbor of $a'$.
This means, though, that the robber can safely return to $a$ for good.
This contradiction implies that $z=a$.
However, it must be that, when two cops reach a neighbor of $a$, which is only possible at $(x,v_2)$ or $(v_2,x)$, the robber has no safe neighbor to move to, which is only possible if two cops also reach the other of $(x,v_2)$ or $(v_2,x)$.
But this requires 8 cops to start with, a contradiction.
Hence $\pc(G)\ge 8$.
\end{proof}

The argument can be sharpened to work for $n\ge 23$, but it is not necessary here to do so; it seems likely that the value $23$ can be reduced even further.

\begin{cor}
\label{c:WheelsNotGraham}
For $n\ge 36$ we have $\pc(W_n\Box W_n) = 2\pc(W_n)\pc(W_n)$.)
\end{cor}


\section{Open Questions}
\label{s:Open}

\begin{qst}
\label{q:grid}
Can the bounds $.178 \approx \frac{5092}{28593} \le \lim_{k,m\rightarrow\infty} \pc(P_k\Box P_m)/km \le .4$ from Theorem \ref{t:CopGrids} be improved?
\end{qst}

Corollary \ref{c:WheelsNotGraham} suggests the following two questions.

\begin{qst}
\label{q:GrahamFam}
Is there an infinite family of graphs $\cG$ for which $\pc(G\Box H)\le \pc(G)\pc(H)$ for all $G,H\in\cG$?
\end{qst}

Theorem \ref{t:GrahamCounter} shows that products of paths is not among such a family.

\begin{qst}
\label{q:GrahamConst}
Is there some constant $a\ge 2$ such that $\pc(G\Box H)\le a\pc(G)\pc(H)$ for all $G$ and $H$?
\end{qst}

In addition to chordal graphs and Cartesian products discussed above, it would be interesting to study other graph classes.
It was proved in \cite{Clarke} that $c(G)\le 2$ for outerplanar $G$, and in \cite{AignFrom} that $c(G)\le 3$ for planar $G$.
Additionally, it was shown in \cite{PisaTan} that if $G$ is a planar graph on $n$ vertices then $\capt_3(G)\le 2n$.

\begin{qst}
\label{q:PlanarOuter}
Are there constant upper bounds on $\pc(G)$ when $G$ is planar or outerplanar? 
If $k=\pc(G)$ then is $\capt_k(G)$ linear?
\end{qst}

Likewise, a result of \cite{PacSneVox} states that every diameter two graph $G$ on $n$ vertices satisfies $\pi(G)\in\{n,n+1\}$.

\begin{qst}
\label{q:Diam2}
Is there a similar, narrow range of values of $\pc(G)$ over all diameter two graphs $G$?
\end{qst}

Finally, Meyniel \cite{FranklGirth} conjectured in 1985 that every graph $G$ on $n$ vertices satisfies $c(G)=O(\sqrt{n})$.
Some evidence in support of this conjecture is found in \cite{BolKumLea}, where it is proved for $G\in\cG_{n,p}$ that when $0<\e<1$ and $p>2(1+\e)\log(n)/n$ we have $c(G) < \frac{10^3}{\e^3}n^{\frac{1}{2}\log(n)}$ almost surely.
(In fact, they also show that when $p\gg 1/n$ we have $c(G) > \frac{1}{(pn)^2}n^{\frac{1}{2}\left(\frac{\log\log(pn)-9}{\log\log(pn)}\right)}$ almost surely.)
Along these lines, we make the following conjecture.

\begin{cnj}
\label{c:Meyniel}
Every graph $G$ on $n$ vertices satisfies $\pc(G)\le 2n/3+o(n)$; i.e., {\rm\ui}$^\sc(G)\ge n/3-o(n)$.
\end{cnj}

\section*{Acknowledgement}
The authors thank the anonymous referees for their very helpful reviews.

\section*{Data Availability Statement}

No data were created or analyzed in this study.


\end{document}